\documentclass[11pt,a4paper,makeidx, amstex]{amsart}
\include{preambolo}

\usepackage[latin1]{inputenc}        
\usepackage[dvips]{graphics}

\usepackage
{graphicx}
\usepackage{epstopdf}
\usepackage{mathrsfs}
\let\mathcal\mathscr
\usepackage{color}		
\usepackage{epsfig}
\psdraft

\usepackage{amssymb}
\usepackage{amsmath}
\usepackage{amsfonts}
\usepackage{latexsym}
\usepackage[T1]{fontenc}
\usepackage[latin1]{inputenc}
\usepackage{hyperref}
\usepackage[all,ps]{xy}
\RequirePackage{amsthm}
\RequirePackage{amssymb}
\usepackage[all,ps]{xy}
\usepackage{latexsym}
\usepackage{amscd}
\usepackage[latin1]{inputenc}
\usepackage[T1]{fontenc}
\usepackage{color}
\RequirePackage{amsthm}
\RequirePackage{amssymb}
\usepackage[all,ps]{xy}
\usepackage{latexsym}

\def\Z{{\mathbb Z}}

\def\Q{{\mathbb Q}}

\def\Mon{\mathop{\rm Mon}\nolimits}

\def\Pic{\mathop{\rm Pic}\nolimits}
\def\rk{\mathop{\rm rk}\nolimits}

\def\tilde{\widetilde}
\def\eps{\varepsilon}
\def\phi{\varphi}

\def\Def{\mathop{\rm Def}\nolimits}

\def\Pic{\mathop{\rm Pic}\nolimits}
\def\dim{\mathop{\rm dim}\nolimits}

\newtheorem{thm}{Theorem}[section]
\newtheorem{defn}[thm]{Definition}

\newtheorem{cor}[thm]{Corollary}
\newtheorem{rmk}[thm]{Remark}
\newtheorem{prop}[thm]{Proposition}

\newtheorem{ex}[thm]{Example}
\newtheorem{conj}[thm]{Conjecture}

\newcommand{\hk}{hyperk\"{a}hler }

\newcommand{\kntiposp}{$K3^{[n]}$ type }

\title[Density of Noether-Lefschetz loci]{Density of Noether-Lefschetz loci of polarized irreducible holomorphic symplectic varieties and applications}

\author{Giovanni Mongardi}
\address{Alma Mater studiorum Universitá di Bologna
Dipartimento di Matematica,
Piazza di Porta San Donato 5,
Bologna, 40126 Italia}
\email{giovanni.mongardi2@unibo.it}
\author{Gianluca Pacienza}
\address{Institut Elie Cartan de Lorraine,
Universit\'e de Lorraine,
B.P. 70239, F-54506 Vandoeuvre-lès-Nancy Cedex
 France}
\email{gianluca.pacienza@univ-lorraine.fr}
\date{\today}

\begin{document}

%
\begin{abstract}
In this note we derive from deep results due to Clozel-Ullmo the density of Noether-Lefschetz loci inside the moduli space of marked (polarized) irreducible holomorphic symplectic (IHS) varieties. 
In particular we obtain the density of Hilbert schemes of points on projective $K3$ surfaces and of projective generalized Kummer varieties in their moduli spaces. 
We present applications to the existence of rational curves on projective deformations of such varieties, to the study of relevant cones of divisors, 
and a refinement of Hassett's result on cubic fourfolds whose Fano variety of lines is isomorphic to a Hilbert scheme of 2 points on a K3 surface. 
We also discuss Voisin's conjecture on the existence of coisotropic subvarieties on IHS varieties and relate it to a stronger statement on Noether-Lefschetz loci in their moduli spaces. 

\end{abstract}
%

\maketitle

%
%
%
{\let\thefootnote\relax
\footnote{
\noindent
\textbf{A.M.S.~classification :} 14C99, 14J28, 14J35, 14J40. \\
\noindent
G.P. was partially supported by the ANR project ``Foliage''. G.M. is supported by the project ``2013/10/E/ST1/00688'' and ``National Group for Algebraic and Geometric Structures, and their Application'' (GNSAGA - INdAM). The project was partially funded by GDRE ``GRIFGA''.
 }}
\numberwithin{equation}{section}


%
\section{Introduction}

Recently Markman and Mehrotra \cite[Theorems 1.1 and 4.1]{MaMe} and Anan'in and Verbitsky \cite{AV} have shown the density, in the corresponding moduli spaces, 
of Hilbert schemes of points on a $K3$ surface and of generalized Kummer varieties. 
The first purpose of this note is to check that the corresponding statement in the polarized case holds true. 
It turns out that a more general polarized density statement can be deduced without much effort from deep results contained in \cite {clull}. 
Precisely, we have the following (see Section 2 for all the relevant definitions and Section 3, Theorem \ref{thm:nl_density}, for a slightly more general statement).

\begin{thm}\label{thm:nl_density-intro}
Let $X$ be an irreducible holomorphic symplectic variety with $\Lambda= H^2(X,\mathbb Z)$ and $H$  a primitive ample line bundle on it. Let $\mathfrak M^0_\Lambda$ be a connected component of the moduli space of $\mathfrak M_\Lambda$  marked polarized deformations of $(X,H)$. 
Let $N\subset \Lambda$ be a sub-lattice of signature $(m,n)$, with $m\leq 1$.  Let us denote by $\mathfrak{D}_N$ be connected  the Noether-Lefschetz locus of points  $t\in\mathfrak{M}^0_\Lambda$ such that $N\subset \Pic({X}_t)$. 
Suppose that $b_2(X)-\mathrm{rank}(N)\geq 3$. Then, if not empty, the  locus $\mathfrak{D}_N$ is  dense in $\mathfrak{M}^0_\Lambda$  with respect to the euclidean topology. 
\end{thm}

Notice that as we require $N\subset \Pic(X)$ and the signature of the Beauville-Bogomolov quadratic form on $\Pic(X)$ is $(1, \rk\Pic(X) -1)$,  the condition on the signature of $N$ is a necessary one. 
By taking the sub-lattice generated by the exceptional class we immediately deduce the following. 

\begin{cor}\label{thm:density} Let $\Lambda$ be a $K3^{[n]}$ lattice (respectively a generalized Kummer lattice). 
Let $\mathfrak M^0_\Lambda$ be a connected component of $\mathfrak M_\Lambda$
containing a marked Hilbert scheme of $n$-points on a $K3$ surface (respectively a marked generalized Kummer variety). 
Let $h\in \Lambda$ be a class with square $(h,h)>0$ and consider  the locus $\mathfrak M^+_{h^{\perp}}\subset \mathfrak M^0_\Lambda$ of points where the class $h$ remains algebraic and belongs to the positive cone.
The locus in
 $\mathfrak M^+_{h^{\perp}}$ consisting of marked pairs $(X,\phi)$ 
such that  $X$ is birational to the Hilbert scheme $S^{[n]}$ for some projective $K3$ surface $S$ (respectively to a generalized Kummer variety $K_n(A)$, for some abelian surface $A$) is dense  in $\mathfrak M^+_{h^{\perp}}$ (in the euclidean topology).
\end{cor}

Although largely expected to be true we believe that these density statements can be very useful in practice. For this reason we think it is convenient to have a general density result
such as Theorem  \ref{thm:nl_density}
that can be easily applied in very different geometrically meaningful contexts. 
As an illustration of this we present several  applications and hope that others will follow. 
The first one concerns the existence of special subvarieties of IHS varieties. 
Precisely we discuss the existence of coisotropic subvariety with constant cycle orbits (see Section 4 for the definitions and the motivation), 
predicted by a conjecture by Voisin \cite{V15}. Using Theorem \ref{thm:nl_density-intro} we observe that to prove a strengthtening of this conjecture 
it is sufficient to check it on a Noether-Lefschetz sublocus in the moduli space (see again Section 4, and in particular Theorem \ref{thm:equiv}, for the precise statements).  

As a by-product of our approach we then provide a proof
of the existence of primitive rational curves whose Beauville-Bogomolov dual lies in (a multiple of) any ample linear system on deformations of Hilbert schemes of points on a $K3$ surface, respectively of generalized Kummer varieties. The primitivity of the curve appears to be significant in light of the recent preprint \cite{OSY} where the authors show that it is not always possible to rule a divisor with rational curves of primitive class.  
The existence of primitive rational curves has already found an application in \cite{MO}. 


We further prove that the density statement obtained in Theorem \ref{thm:nl_density} yields a refinement of a result due to Hassett. Namely we show the density,  
among special cubic fourfolds of any discriminant $d$, of those whose Fano varieties of lines are birational to an Hilbert square of a K3 of fixed degree. 

Then we turn to cones of nef divisors on IHS varieties and deduce from the shape of these cones on certain dense subloci in the moduli space the same information on the whole moduli space. 
This was first observed in \cite{BHTv1} for deformations of Hilbert schemes of points on $K3$'s and works the same for deformations of generalized Kummer varieties.


Finally as a by-product of our Theorem \ref{thm:nl_density} (see Corollary \ref{ex:moduli} for the precise statement) we obtain the density of moduli spaces of sheaves on a K3 (or on an abelian surface) 
inside the Noether-Lefschetz locus  of IHS of $K3^{[n]}$ (or generalized Kummer)-type  possessing a non-zero isotropic class. This is one of the key steps in 
\cite{matsu-isotrop} where Matsushita proves a famous conjecture about the numerical characterization of the existence of a (rational) lagrangian fibration for 
IHS manifolds of $K3^{[n]}$ (or generalized Kummer)-type. 


{\bf Acknowledgments.} We wish to thank O. Benoist, B. Hassett, E. Macr\`i and M. Verbitsky for useful conversations at different stages of this work. We also thank Qizheng Yin for pointing out several inaccuracies in the first version of the paper.  

\section{Preliminaries}
%
For the basic theory of irreducible holomorphic symplectic (IHS) manifolds we refer the reader to \cite{Beauville83,huy_basic}.
Let $X$ be an IHS manifold  and let $\Lambda$ be a lattice such that $H^2(X,\mathbb{Z})\cong \Lambda$. A marking $\phi$ of $X$ is an isometry $\phi:H^2(X,\mathbb{Z})\cong \Lambda$. A marked IHS manifold is a pair $(X,\phi)$, where $X$ is a IHS manifold and $\phi$ a marking of $X$. A symplectic form on $X$ will be denoted by $\sigma_X$. The square of a class $a\in H^2(X,\mathbb Z)$ with respect to the Beauville-Bogomolov quadratic form on $X$ will be denoted by $a^2$.


Recall that the  the {\it positive cone} $\mathcal{C}_X$ is the connected component of the cone of positive classes (with respect to the Beauville-Bogomolov quadratic form) containing a K\"ahler class. Let $\mathfrak{M}_\Lambda^0$ be the connected component of the moduli space of marked IHS manifolds containing $(X,\phi)$. Let $\Omega_\Lambda\cong Gr_{++}^{or}(\Lambda\otimes\mathbb{R})$ be the period domain, which is parametrised by positive oriented two planes inside $\Lambda\otimes\mathbb{R}$. Let 
$$
\mathfrak{p}\,:\,\mathfrak{M}_{\Lambda}^0\rightarrow \Omega_{\Lambda}
$$ be the period map, where $\mathfrak{p}(X,\phi)$ is the positive oriented plane generated by $\phi(\sigma_X+\overline{\sigma}_X)$ and $i\phi(\sigma_X-\overline{\sigma}_X)$. 
Let $h\in \Lambda$ be a class of positive square and consider the sublattice $h^\perp$. Let $\Omega_{h^{\perp}}\subset \Omega_{\Lambda}$ be the set of periods orthogonal to $h$, 
which is isomorphic to $\Omega_{h^\perp}:=Gr_{++}^{or}(h^\perp\otimes\mathbb{R})$. Let $\mathfrak{M}^+_{h^\perp}$ be the set $\{(Y,\psi)\in \mathfrak{M}^0_\Lambda,\,\text{such that } \psi^{-1}(h)\in \mathcal{C}_Y\}$.\\
The restricted period map from $\mathfrak{M}^+_{h^\perp}$ to $\Omega_{h^\perp}$ is surjective, by \cite{huy_basic} and generically injective (this is a direct consequence of Verbitski's global Torelli theorem \cite{Verbitsky}, as the positive cone $\mathcal{C}_Y$ coincides with the cone of K\"ahler classes for very general $(Y,\psi)\in\mathfrak{M}^+_{h^\perp}$, see \cite[Theorem 2.2, item (2) and (4)]{Mark-survey}, \cite[Proposition 5.3]{Mark-survey} and \cite[Corollaries 5.7 and 7.2]{huy_basic}). 
A version of global Torelli that we will use is the following. 
\begin{thm}[\cite{HuyBourb}, Corollary  6.2]\label{thm:corhuy}
Two IHS manifolds $X$ and $X'$ are bimeromorphic if and only if there exists a Hodge isometry $\phi:H^2(X,\mathbb Z)\to H^2(X',\mathbb Z)$ which is composition of maps induced by 
isomorphisms and parallel transport.
\end{thm} 
The local complex structure of $\mathfrak{M}^+_{h^\perp}$ is given by the local deformation space $\Def(Y,\psi^{-1}(h))$ which parametrizes deformations of $Y$ where the class $\psi^{-1}(h)$ remains algebraic.
We have a natural quotient map $\Omega_{h^\perp} \rightarrow \Omega_{h^\perp}/\Mon^2(\Lambda,h):=\mathcal{F}_h$, where $\Mon^2(\Lambda,h)$ is the subgroup of the monodromy group $\Mon^2(\Lambda)\subset O(\Lambda)$ of 
parallel transport operators fixing the class $h\in \Lambda$. Such groups have been determined for manifolds of $K3^{[n]}$ or Kummer type in \cite{Mark-prime} and \cite{Mon2}. This quotient map induces a quotient map $\mathfrak{M}^+_{h^\perp}\rightarrow \mathcal{F}_h$ and moreover the space $\mathcal{F}_h$ is quasi-projective as proven in \cite[Thm. 3.7]{GriHulSan10}. 
\begin{rmk}\label{rmk:fam}
{\em 
In the following, we will sometimes work with the local deformation space $$\Def(X,\psi^{-1}(h)).$$
Using the Torelli theorems above, plus the study of the monodromy group  one gets (cf. \cite[Theorem 1.10]{Mark-survey}) that its
 quotient by the group $\Mon^2(X,h)$ can be considered as an euclidean open subset of the algebraic space $\mathcal{F}_h$. Therefore algebraic families like the Hilbert scheme are well defined in this local setting.\\ }
\end{rmk}
For our purposes it will be relevant also to use deformations of pairs $(X,[C])$, where $X$ is IHS and $C$ is a curve on it, therefore we will make frequent use of the following duality, induced by the quadratic form on $H^2(X)$: we embed $H_{2}(X,\Z)$ into $H^2(X,\Z)$ by the usual embedding of lattices $\Lambda^\vee\rightarrow \Lambda\otimes \Q$ (that is, we use the intersection pairing between $H^2$ and $H_2$ to see an element $[C]$ of the latter as the form $[C]\cdot -$). We call $D$ the dual divisor to a primitive curve $C\in H^2(X,\Q)$ if $D=aC$ for a positive integer $a$ and $D$ is primitive. Conversely, we call $C$ the dual curve to a primitive divisor $D$ if $C=D/div(D)$, where $div(D)$ is the positive generator of the ideal $D\cdot H^2(X,\mathbb Z)$.  




%
\section{Density of Noether-Lefschetz loci}
%
 
Let $\Lambda$ be a lattice of signature $(3,n)$, $n\geq 2$ and let $h\in \Lambda$ be an element of positive square. Let $L$ be $h^\perp$ and let $t\in L$ be a element of negative square. Let $G\subset O(\Lambda)$ be a subgroup of finite index of the orthogonal group $O(\Lambda)$. Let $\Omega_L:=Gr_{++}^{or}(2,L\otimes\mathbb{R})$ be the Grassmannian of positive oriented two planes, which is the period domain associated to the lattice $L$. 
Notice that 
$$
 Gr_{++}^{or}(2,L\otimes\mathbb{R}) = SO(2,n)_{/SO(2)\times SO(n)}
$$
(see e.g. \cite[Sections 1.7 and 2.4]{Verbitsky}).
In this section, we will prove that the set of periods orthogonal to an element in the $G$ orbit of $t$ is dense. In particular, we will apply this result when $\Lambda\cong H^2(X,\mathbb{Z})$ for some IHS manifold $X$. 
This density result in the non polarized setting (that is, inside $\Omega_\Lambda:=Gr_{++}^{or}(2,\Lambda\otimes\mathbb{R}$)), was proved by Anan'in and Verbitsky:
\begin{prop}[Proposition 3.2 and Remark 3.12 \cite{AV}]\label{lem:anve}
Let $T$ be a lattice of signature $(m,r)$, $m\geq 3$ and $r\geq 1$. Let $\Gamma$ be a group of finite index in $O(T)$ and let $\tau$ be an element of $T$.  Let $Gr_{++}^{or}(2,T\otimes\mathbb{R})$ be the Grassmannian of positive oriented two planes.
The set of elements in $Gr_{++}^{or}(2,T\otimes\mathbb{R})$ orthogonal to an element in the orbit of $\tau$ by the group $\Gamma$ is dense in $Gr_{++}^{or}(2,T\otimes\mathbb{R})$. 
\end{prop}
In our situation, one could work with this result and extend it to the polarized case as is done in a special case in \cite[Lemma 3.6]{matsu-isotrop}, however it is more convenient for us to adopt an algebraic approach using powerful results of Clozet and Ullmo \cite{clull}.
Instead of working with the period domain $\Omega_L$, let us work with a quotient of it by an arithmetic subgroup $\Gamma$ (which, in practice, will be the group of isometries of $\Lambda$ fixing $h$ or a finite index subgroup of it like the monodromy group). This variety is isomorphic to 
$$ _{\Gamma\backslash}SO(2,n)_{/SO(2)\times SO(n)}.
$$
The appropriate language to use the aforementioned results is that of Shimura varieties, which is rather separate from the subject of this paper, therefore we will keep it as simple as possible. For the interested reader, a good reference on the topic is \cite{milne2005}. A Shimura variety is obtained from a Shimura datum, which amounts to the following by \cite[Proposition 4.8]{milne2005}:
\begin{defn}
A Shimura datum $(G,D)$ consists of the following:
\begin{itemize}
\item A semisimple algebraic group $G$ defined over $\mathbb{Q}$ of non compact type,
\item A Hermitian symmetric domain $D$,
\item An action of $G(\mathbb{R})^+$ on $D$ defined by a surjective homomorphism $G(\mathbb{R})^+\rightarrow Hol(D)^+$ with compact kernel.
\end{itemize}
\end{defn}
In our setting the datum will be $(SO(2,n), SO(2,n)_{/(SO(2)\times SO(n))})$ with the obvious action of $G$ on $D$, so that the kernel is $SO(2)\times SO(n)$.
\begin{defn}
A connected Shimura variety is defined by the inverse system of quotients $\Gamma\backslash D$, where $\Gamma$ runs over arithmetic subgroups of $G^{ad}(\mathbb{Q})^{+}$ whose pre image in $G(\mathbb{Q})^+$ is a congruence subgroup
\end{defn}
So, for the purpose of density, a statement on the Shimura variety associated to $D$ then works, by continuity of the quotient maps giving the inverse system, on all quotients of $D$ by commensurable arithmetic subgroups.
We will be interested in special Shimura subvarieties, which are those usually called of Hodge type as they are related to variations of Hodge substructures, see \cite[Proposition 2.8]{Moon_linear1}.
\begin{defn}
Given a Shimura datum $(G,D)$, a datum $(H,D_H)$ defines a Shimura subvariety of Hodge type if
\begin{itemize}
\item $(H,D_H)$ is a Shimura datum and $H\subset G$
\item There is a closed immersion between the Shimura varieties associated to the above Shimura data.
\end{itemize}
Under the above hypothesis, all connected components of the closed immersion are of Hodge type.
Moreover, such a subvariety will be called strongly special if there is no intermediate parabolic subgroup between $H$ and $G$.
\end{defn}
Not all varieties of Hodge type have this form, but these suffice for our purposes by \cite[Remark 2.6]{Moon_linear1}. 
\begin{rmk}
{\em {In our situation, we want to consider the Shimura variety associated to the datum
$$(SO(2,n), SO(2,n)_{/(SO(2)\times SO(n))})
$$ and the divisors where a specific class $\lambda$ is kept algebraic, which are associated to the datum 
$$(Stab(\lambda), SO(2,n-1)_{/(SO(2)\times SO(n-1))}),$$ 
where the symmetric domain is given by the orthogonal to $\lambda$. 
Notice that the group $Stab(\lambda)$ is actually maximal parabolic in $SO(2,n)$, so that these divisors of Hodge type are actually strongly special.}}
\end{rmk}
The main result we want to use is the following:
\begin{thm}\cite[Theorem 4.6]{clull}\label{thm:clull}
Let $\mathcal{S}$ be a Shimura variety and let $\mathcal{S}_{n}$ be a sequence of strongly special Shimura subvarieties. Then there exists a subsequence $\mathcal{S}_{n_k}$ and a strongly special Shimura subvariety $\mathcal{M}$ which contains every $\mathcal{S}_{n_k}$ and coincides with their euclidean closure. 
\end{thm}

 As a consequence of this we have the following. 

\begin{prop}\label{prop:ergodic}
Let $\Lambda'$ be a lattice of signature $(3,n),\ n\geq 1$ and let $h\in\Lambda'$ be a class of positive square. Let $L':=h^\perp$ and let $\lambda\in \Lambda'$ be a class which is not in the $O(\Lambda')$-orbit of $h$. 
Let $\Omega_{L'}:=Gr_{++}^{or}(2,L'\otimes\mathbb{R})$ be its period domain. Let $G\subset O(\Lambda')$ be a group of finite index. 
Then the set $\mathcal{D}_{G,\lambda}$ of periods orthogonal to an element in the orbit $G\lambda$ is dense in $\Omega_{L'}$.
\end{prop}
\begin{proof}
Let $\Gamma$ be the subgroup of $G$ fixing $h$, then we have a continuos map 
$$\pi\,:\,\Omega_{L'} \rightarrow \mathcal{F}_{h,\Gamma}:= _{\Gamma\backslash}\Omega_{L'}.$$
Notice that $\mathcal{D}_{G,\lambda}$ is saturated in the fibres of this map, that is $\pi^{-1}(\pi(\mathcal{D}_{G,\lambda}))=\mathcal{D}_{G,\lambda}$. Therefore, it is enough to prove density in $\mathcal{F}_{h,\Gamma}$.
For every element of $\alpha\in G/\Gamma$, we can define a $\Gamma$-orbit of an element $\lambda_\alpha$, where $\lambda_\alpha$ is the projection of $g\lambda$ in $L'$ for some representative $g\in G$ of $\alpha$. 
Notice that $\lambda_\alpha$ depends on the choice of a representative $g$ of $\alpha$, but we will be considering all possible representatives. We can consider the associated divisors $\mathcal{D}_{\Gamma,\lambda_\alpha}$, 
which are (associated to) strongly special Shimura subvarieties as stated before. All points of them correspond to $h$-polarized Hodge structures such that the period is orthogonal to an element in the $G$ orbit of $\lambda$, 
hence their union is the locus we are considering.\\
Notice that $G/\Gamma$ is an infinite set as it is commensurable to the integral part of $O(3,n)/O(2,n)$. Moreover also the $\Gamma$ orbits of all $\lambda_\alpha$ are infinite, 
as the pairing $q(h,g\lambda)$ takes infinitely many values (which implies that the square $q(\lambda_\alpha)$ takes infinitely many values). Therefore we have infinitely many divisors of the form $\mathcal{D}_{\Gamma,\lambda_\alpha}$.
By Theorem \ref{thm:clull} the closure of the infinite union of all the divisors $\mathcal{D}_{\Gamma,\lambda_\alpha}$ can only be contained in a (strongly special) Shimura subvariety
By dimension reasons it has to coincide with the whole space, hence our claim.
\end{proof}

\begin{rmk}{\em{
Notice that, when $G=\Mon^2(\Lambda)$ and $\Gamma=\Mon^2(\Lambda,h)$,
 the space $\mathcal{F}_{h,\Gamma}$ coincides with the (connected) moduli space of polarized IHS manifolds considered in \cite{GriHulSan10}.}}
\end{rmk}

Proposition \ref{prop:ergodic} can then be applied to prove density of the Noether-Lefschetz loci, whose definition is recalled below.

\begin{defn}\label{def:nl}
Let $X$ be an IHS manifold, with $\phi :H^2(X,\mathbb Z)\cong\Lambda$ an isometry, and $H$  a primitive ample line bundle on $X$.
Let $L\subset \Lambda$ be a sub-lattice of signature $(2,m)$ with $m\geq 2$ such that $H\in L^\perp$.
 Let $\mathfrak M_\Lambda$ be the moduli space of marked polarized deformations of $(X,H)$.   Let $\mathfrak M^0_L\subset \mathfrak M_\Lambda$ be a connected component of the moduli space parametrizing marked polarized IHS deformations $(X_t,H_t,\phi_t)$ of $(X,H, \phi)$ 
 such that the period of $X_t$ lies in $L$. 
Let $N \subset \Lambda$  be a sub-lattice of signature $(a,b)$ with $a\leq 1$ and let us denote by $\mathfrak{D}_N$ the locus of $t\in\mathfrak{M}^0_L$ such that there exists a 
primitive embedding $N\subset \Pic({X}_t)$.  
 Then $\mathfrak{D}_N$ is called a Noether-Lefschetz locus.
\end{defn}

\begin{rmk}{\em{
Let $\Lambda$ be the $K3$ lattice. Consider $P:=\langle 2d\rangle$, for some integer $d\geq 2$ and $N:=\langle -2\rangle$. 
Then $\mathfrak D_N$ parametrizes projective $K3$ surfaces of degree $2d$ possessing a $-2$-curve (without requiring a given 
intersection between the $-2$-curve and the polarization). }}
\end{rmk}

\begin{rmk}{\em{The existence of a 
primitive embedding $N\subset \Pic({X}_t)$ is equivalent to the existence of a $g\in O(\Lambda)$ (or in a finite index subgroup $G\subset O(\Lambda)$)
such that $\phi_t^{-1}(g\cdot N)\subset \Pic (X_t)$.  
}}
\end{rmk}

\begin{rmk}{\em{
Notice that the manifolds appearing in $\mathfrak{D}_N$ have periods contained in the  sublocus of  the period domain of dimension at least
$$
 \mathrm{rank}(L)-\mathrm{rank}(N)-2.
$$
Therefore, by the surjectivity of the period map restricted to each connected component of the moduli space \cite{huy_basic}, 
the locus $\mathfrak{D}_N$ is clearly non empty if the above number is positive.}}
\end{rmk}

We are now ready to state and prove our main density result.
\begin{thm}\label{thm:nl_density}
Keep notation as in Definition \ref{def:nl}. Suppose that $N$ is of signature $(a,b)$, with $a\leq 1$ and  that we have $\mathrm{rank}(\Lambda)-\mathrm{rank}(N)\geq 3$. Then, if not empty,  the Noether-Lefschetz locus $\mathfrak{D}_N$ is dense in $\mathfrak{M}^0_L$. 
\end{thm}
As noticed in the Introduction, the condition on the signature of $N$ is a necessary one. 

\begin{proof}[Proof of  Theorem \ref{thm:nl_density}]
 
Let $N\otimes\mathbb{Q}=\langle l_1,\dots,l_{a+b} \rangle$, where $l_i\perp l_j$ for $i\neq j$, $l_i\in \Lambda$. 
Moreover, we can suppose that all $l_i$ apart for $l_{1}$ have negative square. By the condition on the rank of $N$ and its signature, 
notice that $N^\perp\subset \Lambda$ has signature at least $(2,1)$. Let $G$ be any group of  finite index inside $O(\Lambda)$. 
We will prove the result by induction on $a+b$. For $a+b=1$, this is precisely the content of Proposition \ref{prop:ergodic}. 
Let us consider the lattices $M_{r}=\langle l_1,\dots,l_r\rangle$ and let $G_r$ be the stabilizer in $G$ of $M_r$. For every class $[\alpha]$ of $G/G_r$ and every representative $\alpha\in [\alpha]$, 
we have a different projection $R_{r,\alpha}$ of $\alpha\cdot M_r$ inside $L$. Notice again (as in the proof of Proposition \ref{prop:ergodic}) that the $R_{r,\alpha}$'s depend on the choice 
of the representative but we will consider all possible choices. 
Let $\Lambda_{i,\alpha}:=R_{r,\alpha}^\perp\subset L$. 
By the inductive step, periods in the union
$$\cup_{\alpha \in G}\Lambda_{r,\alpha}$$ 
are dense in $\Omega_L$, therefore it suffices to prove that, for any $\alpha$, 
periods in $\Lambda_{r,\alpha}$ orthogonal to an element in the $G/G_{r+1}$ orbit of $M_{r+1}$ are dense. What we are considering is the union of loci of the form
$$D_{r+1,\beta}:=\{P\in \Omega_{\Lambda_{r,\alpha}},\,P\perp g M_{r+1}\,\text{for some }g\in G\,\text{such that }[g]=\beta\}.$$
Here, $\beta\in G/G_{r+1}$. These loci are either divisors in $\Omega_{\Lambda_{r,\alpha}}$ or they are empty if $R_{r+1,\beta}$ is not negative definite. 
As $\langle l_{r+1}\rangle=M_r^\perp\subset M_{r+1}$ is negative definite, this locus is empty if and only if $\Omega_{\Lambda_{r,\alpha}}$ was already empty.
Therefore the density statement we want is precisely the content of Proposition \ref{prop:ergodic} with $\Lambda=M_{r}^\perp$, $L=\Lambda_{i,\alpha}$, $\lambda=l_{r+1}$ and group $G/G_r$, which has finite index in $O(M_{r}^\perp)$ and we are done. 

\end{proof}

\begin{rmk}\label{rmk:puremon2}
{\em {As made clear in the proof the statement of the theorem holds for $G=O(\Lambda)$, but the analogous density statement holds for any finite index subgroup of it, 
like the group of monodromy operators (which has finite index by \cite[Theorem 1.16]{Verbitsky}, see also \cite[Remark 6.7]{HuyBourb}).}}
\end{rmk}

The above theorem has some nice consequences also in the $K3$ case, as an example it can be used to prove that Kummer $K3$ surfaces obtained from an abelian surface of polarisation $(1,d)$ are dense in the moduli space of degree $2e$ polarized $K3$ surfaces, for any $d$ and $e$. Notice the following simpler case.\\

\begin{ex}{\em
Let $\mathcal{F}_4$ be the moduli space of degree $4$ $K3$ surfaces. We denote its elements by $(S,H)$, where $H$ is a nef divisor. By the above theorem, the subset of polarized $K3$ surfaces $(S',L')$ with a nef  class $H'$ of square $4$, and an additional $-2$ curve $E$ such that $H'\cdot E=0$ are dense in $\mathcal{F}_4$. However, notice that the polarization $H$ might not be a combination of $H'$ and $E$, as the proof of \ref{prop:ergodic} uses rational periods. Here, $N=4\oplus -2$, $P=4$ and $L=U^2\oplus E_{8}(-1)^2\oplus -4$. }
\end{ex}

For all the known deformation types of holomorphic symplectic varieties we have several interesting dense Noether-Lefschetz subloci. 
We start with Hilbert scheme of points on $K3$ surfaces and generalized Kummer varieties:

\begin{proof}[Proof of Corollary \ref{thm:density}]
Let $(X,H)$ be a polarized manifold of $K3^{[n]}$-type. Let $N:=\langle -2(n-1) \rangle$ be a rank one lattice. 
Then, by Theorem \ref{thm:nl_density} the Noether-Lefschetz locus $\mathfrak{D}_N$ of $N$ is dense in the deformation of $(X,H)$ (of course 
this follows also immediately from \cite{clull}). We now claim that all points in $\mathfrak{D}_N$
correspond to IHS varieties birational to Hilbert schemes of points.
We can choose a specific embedding of $N$ into $\Lambda:=H^2(X,\mathbb{Z})$ such that $N^\perp$ is unimodular. With such a choice, 
elements in the Noether-Lefschetz locus have the same Hodge structure of an Hilbert scheme of points on a $K3$ surface $S$, where $S$ is the only $K3$ with the Hodge structure of $N^\perp$. 
Thus, if instead of the $O(\Lambda)$ orbit of $N$ we take the $\Mon^2(\Lambda)$ orbit of it, by Remark \ref{rmk:puremon2} and the version of Global Torelli given in Theorem \ref{thm:corhuy}
we get our claim. 

The same proof works, mutatis mutandis, also for generalized Kummer varieties. 
\end{proof}
The above result, in the non polarized case, is the content of \cite{MaMe}.
\begin{rmk}{\em {
In the proof of the non polarized case by Markman and Mehrotra \cite{MaMe}, the authors prove additionally that the locus of actual Hilbert schemes (and not manifolds birational to them) is dense in moduli. Their proof is based on the fact that, for general non projective Hilbert schemes, there are no different birational models of it. In the projective case this is false in general, e.g. take the Hilbert square of a degree two K3 and the Mukai flop on the plane it contains.}}
\end{rmk}
The following result can be proven with the use of Theorem \ref{thm:nl_density} in a way analogous to Corollary \ref{thm:density}, unless the moduli spaces considered are zero dimensional. We also present a different proof.
\begin{cor}\label{ex:moduli}
Let  $\Lambda$ be a lattice either isomorphic to the second cohomology group of a Hilbert scheme of $n$ points on a $K3$ or to a that of generalized Kummer. 
Let $L\subset \Lambda$ a sub-lattice of signature $(1,m)$ with $m\geq 2$  and let $\mathfrak{M}_L$ be the moduli space of manifolds deformation equivalent to the Hilbert scheme of $n$ points on a $K3$ 
or a generalized Kummer which contain primitively $L^\perp$ inside their Picard lattice. Suppose $\mathfrak{M}_L$ is not empty. 
Then the locus of $\mathfrak{M}_L$ corresponding to moduli spaces of sheaves (or their Albanese fibre) on a $K3$ surface (respectively abelian) is dense.
\end{cor}


\begin{proof}
By \cite{Fuj} (see also \cite{Ca83}) points $\mathfrak{M}^{proj}_L$ corresponding to projective IHS are dense in $\mathfrak{M}_L$. Now let $\mathfrak{M}^0_L\subset \mathfrak{M}^{proj}_L$ be a irreducible component corresponding
to marked projective IHS $(X,\phi, h_X)$ such that $\phi (h_X)=h$, for a given positive class $h\in \Lambda$. 
Apply Theorem \ref{thm:nl_density} to  $L':= \langle L^\perp, h\rangle^\perp$, and $N:=\langle -2(n-1)\rangle$ (as in Corollary \ref{thm:density}). This proves the density of Hilbert schemes and a fortiori 
that of moduli spaces of sheaves on $K3$ surfaces. As usual the same proof works for generalized Kummer varieties.
\end{proof}

\begin{proof}[Alternative proof of Corollary \ref{ex:moduli}]
Let us consider the period domain $\Omega_{L}$, where $L:=P^\perp$ inside the appropriate lattice $\Lambda$ (either $U^3\oplus E_8(-1)^2\oplus (-2n+2)$ or $U^3\oplus (-2n-2)$). 
We have a natural (surjective) period map from $\mathfrak{M}_L$ to $\Omega_L$. By a classical results \cite[Proposition 17.20]{Vbook}  
the locus corresponding to manifolds with maximal Picard rank is dense in $\Omega_L$ and so is its preimage in $\mathfrak{M}_L$. 
Let $Y$ be any of these maximal Picard rank manifolds. Regardless of the deformation class, the Picard lattice $N:=\Pic(Y)$ has rank at least $5$ and the discriminant group $N^\vee/N$
 is a finite group with length at most three by elementary lattice theory (its complement in $\Lambda$ has a discriminant group with at most two generators). 
 Therefore, by \cite[Corollary 1.13.5]{nik}, we have that $N=U\oplus N'$ for some $N'$. By \cite[Proposition 4]{A} and \cite[Proposition 2.3]{MW}, 
 this actually implies that these manifolds are moduli spaces of sheaves (or Albanese fibres of them) on a surface. 
 Indeed, the condition in the above cited result is that a specific lattice containing $\Pic(Y)$ contains a copy of $U$, and clearly this is our case.
\end{proof}

\begin{ex}
{\em
Let $X$ be a manifold deformation equivalent to O'Grady's ten dimensional manifold. Let $H\subset \Pic(X)$ be a primitive positive class. Then the locus of manifolds birational to a moduli space of sheaves on a $K3$ surface is dense in the deformations of $(X,H)$. This holds because such resolutions have an extra algebraic class given by the exceptional divisors, hence Proposition \ref{prop:ergodic} applies to the parallel transport of this class on $X$. The equivalent statement holds for the six dimensional O'Grady's manifold. Here, $N=\langle -6 \rangle$ (for a specific choice of an embedding in $\Lambda=U^3\oplus E_8(-1)^2\oplus A_2(-1)$.
}
\end{ex}

Corollary \ref{ex:moduli} is interesting in its own, as many of the applications of density results so far only use it (or at least, can work with it). In the following paragraphs we present two such examples which
seem particularly important to us. 

\subsection{Mori cones}\label{ss:mori-cones}
The goal of the section is to provide an alternative proof of the main result of \cite{BHT}, namely the description of the Mori 
cone of any projective deformation of a $K3^{[n]}$. Notice that 
the strategy we follow here was already presented in \cite{BHTv1}. This strategy could not work due to the lack of the suitable density result. We find it interesting to describe it again, 
as it can analogously lead to the description of the Mori cone of the projective deformations of the O'Grady examples as soon as the Mori cone is known for these 
(see \cite{MZ} for important progress in this direction).  We take the occasion to notice that an analogous statement holds for deformations of generalized Kummers. 
  
The precise result is the following. 

\begin{thm}[\cite{BHT}, Theorem 1, for the $K3^{[n]}$-type]\label{thm:BHT}
Let $(X,h_X)$ be a polarized IHS of $K3^{[n]}$-type (respectively of generalized Kummer type). 
The Mori cone of $X$ has the same description of the Mori cone of a $K3^{[n]}$ (resp. of a generalized Kummer), namely
the Mori cone of $X$in $H^2(X,\mathbb R)_{alg}$ is generated by classes in the positive cone and the image under $\theta^\vee$ of the following
$$
 \{ a\in \tilde\Lambda_{alg}:a^2\geq -2\ ({\textrm respectively}\ a^2\geq 0), |(a,v)|\leq v^2/2, (h_X,a)>0 \}.
$$
\end{thm}
We refer the reader to \cite{BHT} for the relevant definition. Here we want only to stress the r\^ole of density in the proof. 
We recall the following important deformation theoretic result. 

\begin{prop}[\cite{BHTv1}, Proposition 5]\label{prop:def}
 Let $X$ be a projective IHS. Let $R\subset X$ be an extremal rational curve of negative square. 
 Consider a projective family $\pi: \mathcal X\to B$ over a connected curve $B$ with $\pi^{-1}(b)=X$, for a certain $b\in B$, such that the class $[R]$ remains algebraic in the fibers of $\pi$. 
 The the specialization of $R$ in $\pi^{-1}(b_0)$ remains extremal for all  but finitely many $b_0\in B$.
\end{prop}

\begin{proof}[Proof of Theorem \ref{thm:BHT}]
The fact that the above classes are actually in the Mori cone does not depend on density. See \cite[p. 948]{BHT} for its proof. 

For the other inclusion, consider  the rank 2 sublattice $P\subset H^2((K3)^{[n]},\mathbb Z)$ generated by $h$ and by the dual class $R^\vee$ (or rather its saturation). 
As in Definition \ref{def:nl} let $L:=P^\perp$ and  inside the connected component $\mathfrak M^0_L$ containing $(X,h_X)$ consider a general connected curve $B$ passing through $(X,h_X)$.
 
By Proposition \ref{prop:def} an extremal rational curve with negative square remains extremal on the generic point of $B$.  
The generic point of $B$ corresponds to a moduli space of sheaves on a projective K3 surface, as by Theorem \ref{thm:nl_density} moduli spaces of sheaves on projective K3 surfaces are dense in $\mathfrak M^0_L$, hence in $B$.  
By \cite[Theorem 12.2]{BM}, the statement holds for moduli spaces of sheaves on projective K3 surfaces and the desired inclusion follows.  
The proof works {\it verbatim} if $X$ is a projective deformation of a generalized Kummer by replacing \cite[Theorem 12.2]{BM} with \cite[Proposition 3.36]{Yoshi}.
\end{proof}
\subsection{Lagrangian fibrations}\label{sec:lagr}
%

It is conjectured that a non-trivial integral and primitive movable 
(resp. nef) line bundle $L$ on a $2n$-dimensional IHS manifold $X$ with $q_X(L)=0$ induces a rational (resp. regular) Lagrangian fibration. Precisely we should have that 
$L$ is base-point-free, $h^0(X,L)=n+1$ and the morphism
$$
 X\to \mathbb P H^0(X,L)^\vee
$$
is surjective, with connected Lagrangian fibers. 
Several  important results have been obtained in recent years on this problem. 
Matsushita \cite{matsu-top, matsu-add} first proved that the image $B$ of a morphism $f$ from such an $X$ must be a $\mathbb Q$-factorial, klt n-dimensional 
Fano variety of Picard number $1$ and $f$ is a Lagrangian fibration, as soon as $B$ is normal and $0<\dim(B)<2n$. 
The fact that $B$ must be the projective space was proved later by Hwang, under the stronger assumption that $B$ is smooth. 
Bayer-Macr\`i \cite[Theorem 1.5]{BM} (resp. Yoshioka \cite[Proposition 3.36]{Yoshi}) proved the conjecture for moduli spaces of Gieseker stable sheaves on a projective $K3$ (respectively abelian) surface. Independently Markman \cite[Theorem 1.3]{Mark-lagr} proved the conjecture for a general  deformation $X$  of a 
$(K3)^{[n]}$. Markman's result can be extended to any deformation of a 
$(K3)^{[n]}$ thanks to a result due to Matsushita \cite{matsu-def} insuring that if an irreducible holomorphic
symplectic manifold $X$ admits a Lagrangian fibration, then $X$ can be deformed preserving the Lagrangian fibration.
Later Matsushita \cite[Corollary 1.1]{matsu-isotrop} proved that if $X$ is a deformation of $(K3)^{[n]}$ or of a generalized Kummer, any  
non-trivial integral and primitive line bundle $L$ with $q_X(L)=0$ such that $c_1(L)$ belongs to the birational K\"ahler cone of $X$, 
induces a {\it rational} Lagrangian fibration over the projective space. 
His proof uses, among other things three main ingredients: Lagrangian fibrations deform well in moduli \cite{matsu-def};
the conjecture holds on moduli spaces (by \cite[Theorem 1.5]{BM} and \cite[Proposition 3.36]{Yoshi});
 moduli spaces are dense in the Hodge locus of $[c_1(L)]$. The latter is now proved  in \cite[Lemma 3.6]{matsu-isotrop} and can as well be obtained 
 as a particular case of our Corollary \ref{ex:moduli}. Notice that the non-emptyness follows from Markman \cite {Mark-lagr} for deformations of $K3^{[n]}$ and Wieneck \cite{Wie}
for deformations of generalized Kummers.




%
\section{Equivalent conjectures}\label{sec:congetture}
%
Let $X$ be a $2n$-dimensional IHS projective variety. The Chow group $CH_0(X)$ of 0-cycles is non representable by Mumford's theorem (cf. \cite[Chapitre 22]{Vbook}). Nevertheless, by the Bloch-Beilinson conjecture, the $CH_0(X)$ should have an inner structure under the form of a decreasing filtration $F_{BB}^\bullet:=F_{BB}^\bullet CH_0(X)$ satisfying some axioms (see \cite[Chapitre 23]{Vbook}). While this conjecture appears to be out of reach,   Beauville, inspired by the multiplicative splitting on the Chow ring of abelian varieties \cite{Babe} and by the case of $K3$ surfaces \cite{BV},   suggested in \cite{B07} to investigate an interesting consequence of a (conjectural) splitting of this filtration, called ``weak splitting property''. This property consists in the injectivity of the cycle-class map when restricted to the sub-algebra generated by classes of divisors. 
This conjecture of Beauville gave rise to several works in the last years \cite{V08, Vrat, Fer11, Fu, Rie, Lat, FLV, SYZ, SY, Yin}. Very recently cf. \cite{V15}, Voisin developed a different approach to the study of the filtration $F_{BB}^\bullet$ and its conjectural splitting. 
For any integer $1\leq i\leq n$, she considers 
$$
S_i(X):\{x\in X: \dim O_x\geq i\},
$$
where $O_x$ is the orbit of $x$ under rational equivalence. Notice that a subvariety $Y$ of such an orbit is a {\it constant cycle} subvariety of $X$ (cf. \cite{Huy14}), i.e. a subvariety whose points are all rationally equivalent in $X$. 
Using Mumford's theorem, one can show that any of the (possibly countably many) irreducible components of $S_i(X)$ has dimension $\leq 2n-i$. Then Voisin defines
$S_i CH_0(X)\subset CH_0(X)$ to be the subgroup generated by classes of points in  $S_i(X)$. In this way she obtains a descending filtration $S_\bullet CH_0(X)$ on $CH_0(X)$ and she conjectures that it should be opposite to the Bloch-Beilinson filtration and thus provides a splitting of it, in the sense that, for any $i=1,\ldots, n$
$$
 S_i CH_0(X) \cong CH_0(X)/ F^{2n-2i+1}_{BB}CH_0(X).
$$
In this direction an important r\^ole is played by the following
\begin{conj}[\cite{V15}, Conjecture 0.4]\label{conj:voisin}
Let $X$ be a $2n$-dimensional holomorphic symplectic variety. For any $i=1,\ldots,n$ there exists a component $Z$ of $S_i(X)$ of maximal dimension $2n-i$. 
\end{conj}
She then observed that if Conjecture \ref{conj:voisin} holds (and if of course the Bloch-Beilinson filtration exists), then the map
$$
 S_i CH_0(X) \to CH_0(X)/ F^{2n-2i+1}_{BB}CH_0(X)
$$
is surjective. 
Back to Conjecture \ref{conj:voisin}, Voisin observed that if $Z\subset S_i(X)$ has maximal dimension $2n-i$ then $Z$ is swept by $i$-dimensional constant cycle subvarieties, which are the orbits $O_z$ of its points $z\in Z$.
Conjecture \ref{conj:voisin} has been proved in the following cases: for $i=2$ and $X$ a very general double EPW sextic, \cite{Fer11}; for $i=n$ and $X$ having a Lagrangian fibration (\cite{Lin}); for a generalized Kummer and any $i$ (\cite{Lin2}); for the Fano variety of line on a cubic 4fold and the LLSV 8fold and any possible $i$, \cite{V15}; for moduli spaces of stable objects on a $K3$ surface, \cite{SYZ}; for $i=1$ when $X$ is deformation equivalent to the punctual Hilbert scheme of a $K3$ surface (respectively when  $X$ is deformation equivalent to a generalized Kummer) in \cite{CP} (resp.  in \cite{MP}).

To state our results, let us first define the following:
\begin{defn}
Let $X$ be an IHS projective variety, let $H$ be a divisor of positive square on $X$ and let $Z\subset X$ be a subvariety of pure codimension $i$. 
\begin{itemize}
\item[(i)] $Z$ is called a Voisin coisotropic subvariety if $Z\subset S_i(X)$.
\item[(ii)] A Voisin coisotropic subvariety $Z\subset X$ is said to have RCC orbits if the orbits $O_z$ of its points with respect to rational equivalence are rationally chain connected. 
\item[(iii)] If a Voisin coisotropic subvariety $Z\subset X$ has RCC orbits, these are called  
of type $m(H^\vee)$, for a certain integer $m>0$, if, for a general point $z\in Z$, any two points of $O_z$ are connected by a chain of rational curves of class $m(H^\vee)$.
\end{itemize}
\end{defn}
Here, the curve class $H^\vee$ is the class $H/div(H)$ under the embedding $H_2(X,\mathbb{Z})\rightarrow H^2(X,\mathbb{Q})$ given by lattice duality, and the divisibility $div(H)$ is the positive generator of $H\cdot H^2(X,\mathbb{Z})$.

The aim of this section is to relate  (a strengthening of) Conjecture \ref{conj:voisin}  to an existence conjecture on Noether-Lefschetz loci. 

The main tools will be Theorem \ref{thm:nl_density} and an easy, yet useful density principle
which we state and prove below. This principle simply says that to prove the existence of Voisin's coisotropic subvariety with (good) 
RCC orbits of a given type it is sufficient to have existence on a dense subset of the relevant moduli spaces.

%

To make things precise we introduce some terminology and notations.
Given a polarized IHS variety $(X,h)$ we will consider the moduli space 
of genus zero stable maps
$\overline{ M}_0(X, [h]^\vee)$ of class $[h]^\vee\in H_2(X,\mathbb Z)$. If $M$ is an irreducible component of 
$\overline{ M}_0(X, [h]^\vee)$ we will denote by $C\to M$ the universal curve above it and consider the natural evaluation morphism $ev:C\to X$.

 \begin{thm}\label{thm:densityprinciple}
Let $1\leq k\leq n$ be an integer.
Suppose there exists a subset $\mathfrak D\subset \mathfrak M^+_{h^{\perp}}$ which is dense with respect to the euclidean topology and
such that for all $t\in \mathfrak D$: 
\begin{enumerate}
\item[(i)] there exists an irreducible component $M_t$ of the moduli space of genus zero stable maps
$\overline{ M}_0(X_t, [h_t]^\vee)$ of dimension $2n-2$ and 
\item[(ii)] the image of evaluation morphism $ev_t:C_t\to X_t$ has dimension $2n-k$. 
\end{enumerate}
Then  any $X$ in $\mathfrak M^+_{h^{\perp}}$ contains a Voisin coisotropic subvariety $Z$ of codimension
$k$ with RCC orbits of type (a multiple of) $[h_X]^\vee$. 
\end{thm}
\begin{defn}
The RCC orbits ruled by rational curves verifying items (i) and (ii) of Theorem \ref{thm:densityprinciple} will be called good.
\end{defn}
\begin{rmk}\label{rmk:good}
\rm{Using e.g. \cite[Proposition 2.1]{OSY} one can check that the RCC orbits of Voisin's coisotropic subvarieties of type $H$, where $H$ is a primitive divisor, are good. }
\end{rmk}

\begin{ex}
{\em Let $(S,H)$ be a very general K3 surface of degree $4$ and let $C\in |2H|$ be a general curve with only 5 nodes as singularities (that is, a curve of geometric genus 4). Let $R$ be the rational curve inside $S^{[5]}$ given by a $\mathfrak{g}^1_5$ on the normalization of $C$. Its class is $2H^\vee-8\tau_5$ and, as we let the curve $C$ and the linear series on it vary, we obtain exactly a 8 dimensional family of such curves, which can be proven analogously to \cite[Prop. 3.6]{KLM} using the fact that, under suitable generality assumptions, the curve normalization of $C$ is a Brill-Noether general curve by \cite[Cor. 8.5]{CFGK}. Therefore, the locus $Z$ they cover is a Voisin's coisotropic subvariety of type $2(H^\vee-4\tau_5)$ with good RCC orbits. 
}
\end{ex}

\begin{proof}[Proof of Theorem \ref{thm:densityprinciple}] 
In order to prove the theorem it is sufficient to prove that 
given any holomorphic symplectic variety  $X$ in $\mathfrak M^+_{h^{\perp}}$  the conclusion holds on (a contractible open subset of) the subset $B:=\Def(X)_{h^\perp}$ of the Kuranishi space $\Def(X)$ of deformations of $X$ parametrizing those where the class of $h$ remains algebraic. We will first show it on an open subset $U\subset B$ and then derive the conclusion on the whole $B$.

Let $\pi :\mathcal X\to B$ be the universal family.  Consider the relative moduli space of genus zero stable maps 
$\overline {\mathcal M}_0(\mathcal X/B, [h]^\vee)$. By abuse of notation we denote by $[h]^\vee$ the class of the section of the local system $R^{4n-2}\pi_*\mathbb Z$ whose value at the point $b\in B$ is the class in $H_2(X_t,\mathbb Z)$ dual to $\phi_t^{-1}(h)$. By hypothesis $\overline {\mathcal M}_0(\mathcal X/B, [h]^\vee)$ has dense image in the base $B$. Since it is a scheme of finite type, there exists an irreducible component $\mathcal M$ dominating the base and such that the restrictions $\mathcal M_{|b}$ over the points $t\in \mathfrak D$ contain the components $M_t$ given by the hypothesis of the theorem. Denote by $\mathcal C \to \mathcal M$ the universal curve and by $ev:\mathcal C \to \mathcal X$ the evaluation morphism over $B$. 

Consider the set
$$
B_{bad}:=\{ b\in B: \mathcal M_{|b}\ {\rm{is\  reducible}}\}.
$$
By \cite[Th\'eor\`eme 9.7.7]{FGA Explained} we have either 

(a) $B\setminus B_{bad}$ contains an open subset ;

or 

(b) $B_{bad}$ contains an open subset.

In case (a), let $U'\subset (B\setminus B_{bad})$ be the open subset. 
By definition for all $b\in U'$ we have that $\mathcal M_{|b}$ is irreducible. 

Inside $U'$ consider the open sublocus $U$ of points $b$ 
where the rank of the evaluation morphism $ev_b$ restricted to $\mathcal C_{|b}$ is maximal {\it and} the dimension of $\mathcal M_{|b}$ is constant. 
By density $U\cap \mathfrak D\not= \emptyset$. Therefore we have that 
$$
 \dim(\mathcal M_{|b})=\dim(M_{|b})=2n-2
$$ 
and 
$$
 \rk (ev_b)=2n-k.
$$
Set $Z_b:=ev_b(\mathcal C_{|b})$. 
By dimension count, for all $b\in U$, through the general point of  $Z_b$ we have
a $k$-dimensional RCC subvariety (contained in $Z_b$) and the theorem is proved over $U$
in this case. 

In case (b), one proceeds {\it mutatis mutandis} in a similar way. 


In both cases by construction any two points of $Z_b$ can be joined by a rational curve of the same class. 

To conclude the proof of the theorem, let $X_0$ in $\Def(X)_{h^\perp}\setminus U$.
Let $T\subset \Def(X)_{h^\perp}$ be a  curve passing through $X_0$ and not contained in $\Def(X)_{h^\perp}\setminus U$. Up to shrinking $T$ we may suppose that $(T\setminus [X_0])\subset U$. Define $Z_0\subset X_0$ to be the limit, for $t\in (T\setminus [X_0])$,
 of the subvarieties $Z_t\subset X_t$ having dimension $2n-k$ and covered by $k$-dimensional RCC subvarieties, whose existence has been shown before. Let $x_0\in Z_0$ be a point and $\{x_t\in X_t\}_{t\in  (T\setminus [X_0])}$ a set of  points converging to it. 
 Let $F_t\subset X_t$ be a $k$-dimensional RCC subvariety containing $x_t$ and
 let $F_0$ be the limit of the $F_t$'s. It is RCC, as limit of RCC's. Therefore also $X_0$ contains a $(2n-k)$-dimensional subvariety $Z_0\subset X_0$ which is covered by $k$-dimensional RCC subvarieties and the theorem is proved. 

\end{proof}


We state the following conjectures, the first one being a slight strengthening of Voisin's original conjecture.

\begin{conj}\label{conj:voisin2}
Let $X$ be a $2n$-dimensional IHS projective variety. For any $i=1,\ldots,n$ there exists a  primitive positive divisor $H_i$, a positive integer $m_i>0$ and a codimension $i$ Voisin's coisotropic subvariety $Z_i\subset X$ with good RCC orbits of type $m_i(H_i)^\vee$.
\end{conj}

The advantage of coisotropic subvarieties with RCC orbits is that we control easily their degenerations, while fixing the type allows to deal with a parameter space (of stable genus zero
maps) which will be of finite type. 

\begin{rmk}\label{rmk:oberdieck}
\rm{The recent preprint \cite{OSY} shows that in certain cases the integers $m_i$ can be strictly greater than one.
}
\end{rmk}

\begin{conj}\label{conj:voisinNL}
Let $\mathfrak{M}^+_{h^\perp}$ be the moduli space of polarized deformations of a projective IHS variety $X$, $\mathfrak{D}_N\subset \mathfrak{M}_H$ the  Noether-Lefschetz locus corresponding to a lattice $N$ such that $b_2(X)-\mathrm{rank}(N)\geq 3$ and $U\subset \mathfrak{D}_N$ a dense subset. Then, for every $t\in U$ and every $1\leq i \leq dim(X)/2$, 
there exist an integer $m>0$ and a codimension $i$ Voisin's coisotropic subvariety $Z_{t,i}\subset X$ with good RCC orbits of type $m (H_{t})^\vee$, where  $H_t\in \Pic({X}_t)$ is the primitive polarization such that $\phi_t(H_t)=h$.
\end{conj}

\begin{conj}\label{conj:voisinNLforte}
Let $\mathfrak{M}^+_{h^\perp}$ be the moduli space of polarized deformations of a projective IHS variety $X$, $\mathfrak{D}_N\subset \mathfrak{M}_H$ the  Noether-Lefschetz locus corresponding to a lattice $N$ such that $b_2(X)-\mathrm{rank}(N)\geq 3$ and $U\subset \mathfrak{D}_N$ a dense subset. Then, for every $t\in U$, for every  divisor $P_t\in \Pic({X}_t)$ of positive square and every $1\leq i \leq dim(X)/2$, there exist an integer $m_i>0$ and  a codimension $i$ Voisin's coisotropic subvariety $Z_{t,i}\subset X$ with good RCC orbits of type $m_i (P_{t})^\vee$.
\end{conj}

\begin{conj}\label{conj:voisinfortissima}
Let $\mathfrak{M}^+_{h^\perp}$ be the moduli space of polarized deformations of a projective IHS variety $X$. 
Then, for every $t\in \mathfrak{M}^+_{h^\perp}$, for every positive divisor $P_t\in \Pic({X}_t)$ and every $1\leq i \leq dim(X)/2$, there exist an integer $m_i>0$ and a codimension $i$ Voisin's coisotropic subvariety $Z_{t,i}\subset X$ with good RCC orbits of type $m_i (P_{t})^\vee$.
\end{conj}

\begin{thm}\label{thm:equiv}
Let $X$ be a projective IHS variety with $b_2(X)\geq 4$. Then the 4 above conjectures \ref{conj:voisin2}, \ref{conj:voisinNL},  \ref{conj:voisinNLforte} and \ref{conj:voisinfortissima} are equivalent. 
\end{thm}

\begin{proof}[Proof of Theorem \ref{thm:equiv}]
Clearly, Conjecture \ref{conj:voisinfortissima} implies Conjecture \ref{conj:voisinNLforte} which implies Conjecture \ref{conj:voisinNL}. Let us show that Conjecture \ref{conj:voisinNL} implies Conjecture \ref{conj:voisin2}. 
Manifolds in the Noether-Lefschetz locus $\mathfrak{D}_N$ are dense in the moduli space $\mathfrak{M}^+_{h^\perp}$ by Theorem \ref{thm:nl_density}. 
 As Conjecture \ref{conj:voisinNL} holds on $\mathfrak{D}_N$ , we can apply Theorem \ref{thm:densityprinciple} to obtain Conjecture \ref{conj:voisin2}. Finally, let us prove that the first Conjecture implies the last. Again, this is a simple corollary of Theorem \ref{thm:densityprinciple}. Indeed, let $H\in \Pic(X)$ be any primitive positive class and let $(Y,H_Y)$ be a very general deformation of $(X,H)$. By Conjecture \ref{conj:voisin2}, we have a rational curve of class $mH^\vee$ which connects any two points in a general fibre of a codimension $i$ coisotropic variety. As $(Y,H_Y)$ is very general, the only  curve classes are given by the multiples of $H_Y^\vee$ and thus we can apply Theorem \ref{thm:densityprinciple} to obtain the result for $(X,H)$.
\end{proof}


\section{Existence of rational curves via density}

In this section we will use linear series on surfaces to construct a dense set of points corresponding to IHS containing a rational curve for the moduli spaces of pairs of deformations of \kntiposp or generalized Kummer varieties type. By Theorem \ref{thm:densityprinciple}, this will be enough to prove that all such IHS contain a rational curve whose Beauville-Bogomolov dual class is (a multiple of) the polarization. The result of this section is the following:
\begin{thm}\label{onethmtounirulethemall}
Let $(X,H)$ be a polarized manifold of \kntiposp or of Kummer type, with $H$ ample and primitive. Then there exists a rational curve whose class is dual to $|H|$.
\end{thm}
As it is convenient when working with this deformation classes, we will use the index $\epsilon\in\{0,1\}$ to distinguish between them. Therefore in the following $S_{\epsilon}$ and $S^{[n]}_\epsilon$ will be a $K3$ surface and its Hilbert scheme of $n$ points when $\epsilon=0$ or an abelian surface and its generalized Kummer when $\epsilon=1$. When $H\subset S_\epsilon$ is a divisor, we will denote with $\{H\}$ the connected component of $Hilb(S_\epsilon)$ containing $|H|$.
 We will use nodal curves, therefore we use the following result from \cite{CK} and \cite{KLMcurve}. 

\begin{prop}\label{prop:seriesbound}
Let $(S_\epsilon,H)$ be a general  polarized $K3$ or abelian surface of genus $p:=p_a(H)$. Let $\delta$ and $n$ be integers satisfying $0 \leq \delta \leq p-2\epsilon$ and $n+\epsilon \geq 2$.  Then the following hold:
  \begin{itemize}
\item [(i)] There exists a $\mathfrak{g}^1_{n+\epsilon}$ on the normalization of a curve in $\{H\}$ with $\delta$ nodes as singularities if and only if 
\begin{equation} \label{eq:boundA}
\delta \geq \alpha\Big(p-\delta-\epsilon-(n-1+2\epsilon)(\alpha+1)\Big), 
\end{equation}
where
\begin{equation} \label{eq:alpha}
\alpha= \Big\lfloor \frac{p-\delta-\epsilon}{2(n-1+2\epsilon)}\Big\rfloor; 
\end{equation}
\item [(ii)]  whenever nonempty, the scheme of these linear series is equidimensional of dimension  $\min\{p-\delta,2(n-1+\epsilon)\}$.
\end{itemize}
\end{prop}

We have a natural map from the $\mathfrak{g}^{1}_{n+\epsilon}$ on the curve to the Hilbert scheme $S^{[n+\epsilon]}$ which, up to translation, lands in $S_{\epsilon}^{[n]}$. The class of this rational curve $R_{n}^{p-\delta}$ in $S^{[n]}_\epsilon$ is computed in \cite[Lemma 3.3]{KLM} and is equal to $H-(p-\delta+n-1+\epsilon)\tau_n$, where $\tau_n$ is the class of a fibre of the exceptional divisor $2\Delta_n$ of the Hilbert-Chow morphism. Let $e=GCD(2n-2+4\epsilon,p-\delta+n-1+\epsilon)$, $f=(p-\delta$ and $d=(2n-2+4\epsilon)/e$. Then the class $dH-f\Delta_n$ is the class of a divisor whose dual curve is $R_{n}^{p-\delta}$.



Let us construct enough curves to obtain a dense subset of all moduli spaces of pairs:
Let $g=p-\delta$ be the geometric genus and let $k=g-\epsilon$ modulo $2(n-1+2\epsilon)$ be in the interval $[0,2(n-1+2\epsilon)-1]$. Let $\alpha=(g-\epsilon-k)/(2n-2+4\epsilon)$ and let $\delta_{min}$ be $\frac{(g-\eps-k)}{2}(\alpha-1)+k\alpha$.
\begin{prop}\label{prop:conticurve}
Keep notation as above, then for all $g\geq n$, all $r\in\mathbb{N}$ and all pairs $(S_{\eps},H)$ of genus $g+r+\delta_{min}$ there exists a curve in $\{H\}$ with $\delta_{min}+r$ nodes, geometric genus $g$ and a $\mathfrak{g}^1_{n+\epsilon}$ on its normalization such that the associated rational curve $R_{g,\delta_{min}+r}$ in $S_{\epsilon}^{[n]}$ has class $H-((g-\epsilon)+(n-1+2\epsilon))\tau_n$ and square $(2r-2+2\epsilon)-\frac{(n-1+2\epsilon-k)^2}{2(n-1+2\epsilon)}$.
\end{prop}
\begin{proof}
The existence statement is clear, as $\alpha$ is as in \eqref{eq:alpha} and $\delta_{min}$ is the minimal number of nodes satisfying \eqref{eq:boundA}, therefore this is a direct consequence of Prop. \ref{prop:seriesbound}. If we write the genus of $(S,H)$ as $(g-\epsilon-k)+r+\delta_{min}+\epsilon+k$, standard algebraic computations give us that 
$$g+\delta_{min}=\frac{(g+\epsilon+n-1)^2)-k^2}{4(n-1+2\epsilon)}-\frac{n-1+2\epsilon-2k}{4}+\epsilon.$$ 
Therefore the square of $R_{g,\delta_{min}+r}$ is

$$q(R_{g,\delta_{min}+r})=(2r-2+2\epsilon)-\frac{(n-1+2\epsilon-k)^2}{2(n-1+2\epsilon)}.$$
\end{proof}
From the class of the above curve, it is clear that the divisibility of the dual divisor $D_{g,r,\delta_{min}}$ is determined by the integer $k$: $H-((g-\epsilon)+(n-1+2\epsilon))\tau_n=H-(\alpha+1/2)\Delta_n-k\tau_n=L-k'\tau_n$, where $L$ is a class in $H^{2}(S^{[n]}_\epsilon)$ and $-n+1-2\epsilon \leq k'<n-1+2\epsilon$. Hence the dual divisor has divisibility $t$, where $t$ is the order of $k'$ modulo $2n-2+4\epsilon$.
\begin{prop}\label{prop:sisondensi}
The set of pairs $(S^{[n]}_\epsilon,D_{g,r,\delta_{min}})$ with $D_{g,r,\delta_{min}}$ dual to the curves $R_{g,\delta_{min}+r}$ of Prop. \ref{prop:conticurve}, $D_{g,r,\delta_{min}}^2=2d$ and $div(D_{g,r,\delta_{min}})=t$ is dense in all connected components of $\mathcal{M}_{2d,t}$. 
\end{prop}
\begin{proof}
We wish to consider the curves $R_{g,\delta_{min}+r}$. As proven in \cite[Thm. 2.4]{CP} for Hilbert schemes and \cite[Thm. 4.2]{MP} for Kummers, for every connected component of $\mathcal{M}_{2d,t}$ there exists $k'\in [-n+1-2\epsilon,n-1+2\epsilon]$ and $(S_{\epsilon},H')$ polarized surface such that $(S^{[n]}_\epsilon, tH'+k'\Delta)$ is in the desired connected component of $\mathcal{M}_{2d,t}$. Clearly, we can replace $tH'$ with $tL$ for any $L\in H^2(S^{[n]},\mathbb{Z})$ with $L^2=H'^2$ as $tL+k'\Delta$ will have the same monodromy invariant of $tH'+k'\delta$. Therefore, for any $t,d$ and any connected component of $\mathcal{M}_{2d,t}$, there are $r$ and $S_{\epsilon}$ such that $(S_{\epsilon}, D_{g,r,\delta_min})$ is in the desired component, and this happens for countably many $g\geq n+\epsilon$. 
\end{proof}

Summing all of these, we are ready to prove Theorem \ref{onethmtounirulethemall}.
\begin{proof}
Let $(X,H)$ be a polarized pair (of the appropriate deformation types) and let $M\subset \mathcal{M}_{2d,t}$ be its component of the moduli space of pairs. By Prop. \ref{prop:sisondensi}, there is a dense subset of $M$ whose points have a family of rational curves of class dual to $H$. As the moduli space of stable rational curves is closed, there are rational curves on every element of $M$. 
\end{proof}

\begin{rmk}\label{rmk:oberdieck2}
\rm{In \cite[Corollary A.3]{OSY} the authors discovered a numerical condition insuring, for projective deformations of $K3^{[n]}$s, the existence of a uniruled divisor ruled by rational curves having primitive class. The condition 
is also necessary if the curves are irreducible (e.g. at the general point in the corresponding moduli space). Nevertheless the numerical condition holds at most in a finite number of cases in each dimension $2n$ 
(even letting the degree of the polarization vary). Therefore in these sporadic cases, the primitive rational curves that we construct in Theorem \ref{onethmtounirulethemall}  must then cover coisotropic sub-varieties of codimension > 1.
}
\end{rmk}

For higher-dimensional subvarieties the deformation theory can either be understood in the smooth case (cf. \cite{LP} which generalizes the Lagrangian case, done in \cite{Vstab}), but it is not sufficient to conclude the existence for a general $X$, or it is very difficult to control in the singular case (see \cite{Lehn} for some partial results in this direction). 
Our hope is that, with some further work, this new approach via density can be successfully used to obtain the existence of 
constant cycle isotropic subvarieties of dimension $\geq 2$, a problem which seems to be a challenging one.



%
\section{Fano varieties of lines of special cubic fourfolds}
%

The goal of the section is to observe that it is possible to deduce immediately from Theorem \ref{thm:nl_density} a generalization of some results obtained by Hassett on the Fano variety of lines of special cubic fourfolds. We need to recall the basic definitions and results.

A cubic fourfold is smooth cubic fourfold hypersurface in $\mathbb P^5$.
 We consider the coarse moduli space $\mathcal C$  parametrizing cubic fourfolds.
Following Hassett \cite{Hass} a cubic fourfold $X$  is said to be {\it special} if it contains an algebraic surface not homologous to a complete intersection. We collect many of Hassett's result in the following statement. 
\begin{thm}\label{thm:atuttohass}
\begin{enumerate}
\item[(i)](see \cite[Theorem 3.1.2 and Proposition 3.2.4]{Hass}) A cubic fourfold $X$ is special (of discriminant $d$) if and only if the lattice
$H^4(X,\mathbb Z)\cap H^{2,2}(X)$ contains a primitive lattice 
of rank 2 and discriminant $d$,  which contains the
class $h^2$, where $h$ denotes the hyperplane class. 
\item[(ii)](see \cite[Theorem 4.3.1]{Hass}) Let $d\geq 8$ be an integer. The set $\mathcal C_d$ of special cubic fourfolds of discriminant $d$ is not empty iff $d\equiv 0,2 (\textrm{mod}\ 6)$, 
\item[(iii)](see \cite[Theorem 3.2.3]{Hass}) Assume $d\geq 8$ is an integer : $d\equiv 0,2 (\textrm{mod}\ 6)$. Then the set $\mathcal C_d$ is an irreducible algebraic divisor of $\mathcal C$.
\end{enumerate}
\end{thm}

On the other hand it is well known, thanks to Beauville and Donagi \cite[Proposition 2]{BD}, that the Fano variety of lines $F(X)$ on a cubic fourfold $X$ is an IHS variety deformation equivalent to the Hilbert scheme of two points on a K3 surface. Varying the cubic fourfold we get a complete family of such deformations, {\it i.e.} a whole connected component of the relevant moduli space. 

Moreover they showed (see \cite[Proposition 4]{BD}) that the natural Abel-Jacobi map 
yields an isomorphism of  Hodge structures
$$
 H^4(X, \mathbb Z) \to H^2 (F(X), \mathbb Z).
$$
More precisely (see \cite[Proposition 6]{BD})
we have an isomorphism of polarized Hodge structures
\begin{equation}\label{eq:iso}
 H^4(X, \mathbb Z) (-1)_{h^2} \to H^2 (F(X), \mathbb Z)_{g}
\end{equation}
where $g$ is the class of the hyperplane section of $F(X)$ in the Pl\"ucker embedding and $H^2 (F(X), \mathbb Z)_{g}$ (respectively $H^4(X, \mathbb Z) (-1)_{h^2}$) denotes the classes orthogonal to $g$ 
(respectively the classes orthogonal to $h^2$ in $H^4(X, \mathbb Z)$ endowed with the opposite sign of the intersection form).
Finally recall that by Voisin \cite{VoiTor} we know that the natural period map for cubic fourfold is a open immersion. 

The question is to understand when $F(X)$ is {\it isomorphic to}  (and not only deformation of) the Hilbert scheme of 2 points on a $K3$. 
Hassett proved the following necessary condition.

%
%
%
%

\begin{prop}[Proposition 6.1.3, \cite{Hass}]\label{prop:hass}
 Assume that the Fano variety of a generic special cubic
fourfold of discriminant d is isomorphic to $S^{[2]}$ for some K3 surface $S$. Then
there exist positive integers $m$ and a such that $d = 2\frac{m^2+m+1}{
a^2}$.
 \end{prop}
 
Then he also obtained the following sufficient condition. 
 \begin{thm}[Theorem 6.1.4, \cite{Hass}]
Assume that $d = 2(m^2 + m + 1)$ where $m$ is an integer $\geq 2$.
Then the Fano variety of a generic special cubic fourfold X of discriminant
$d$ is isomorphic to $S^{[2]}$, where $S$ is a K3 surface.
 \end{thm}
 
 Our main result is the following. 
 
 \begin{thm}\label{thm:hass++}
For every integer $d$ such that the set $\mathcal C_d$ of special cubic fourfolds  of discriminant
$d$ is not empty, those whose
 Fano variety of lines is birational to a $K3^{[2]}$, of a fixed degree $2e_0$, are dense in the euclidean topology.
 \end{thm}
\begin{proof}
By Voisin's Torelli theorem and the isomorphism of polarized Hodge structures (\ref{eq:iso}) we can see $\mathcal C_d$ as a divisor in the period domain $\Omega_\Lambda$, where $\Lambda\cong H^2((K3)^{[2]},\mathbb Z)$. By \cite[Proposition 3.2.4]{Hass}, up to automorphisms fixing $h^2$, there exists a unique primitive sublattice $K\subset H^4(X,\mathbb Z)$ 
of rank 2 and discriminant $d$,  which contains the
class $h^2$. Let $\tilde K\subset H^2(F(X),\mathbb Z)$ be its image via the isomorphism (\ref{eq:iso}). 
 By Theorem \ref{thm:atuttohass}, items (ii) and (iii),  
$\mathcal C_d$ corresponds, via the period map $\mathfrak M_\Lambda\to \Omega_\Lambda$,  to the image of 
a connected component $\mathfrak M^0_L$, where $L$ is $(\tilde K)^\perp$. Now take $N\subset H^2((K3)^{[2]},\mathbb Z)$ to be the rank 2 sublattice generated by the exceptional class $\mathfrak e$ and by a class $h_S$ orthogonal to $\mathfrak e$ and of square $2e_0$. Notice that $\mathfrak D_N$ is the locus of manifolds birational to $S^{[2]}$, where $S$ is a  $K3$ surface having a positive class of degree $2e_0$.
Then by Theorem \ref{thm:nl_density} the (non-empty) set $\mathfrak D_N$ is dense in $\mathfrak M^0_L$ and we are done.
\end{proof}
\begin{rmk}
\em{
In \cite[Proposition 5.18]{DM}, a similar density statement was proved, namely that there are countable degrees $2e$ for a general $K3$ surface $S$ such that $S^{[2]}\cong F(X)$ for some cubic fourfold $X$ and these cubic fourfolds are dense in the euclidean topology. This result differs from ours in two ways: it is stronger as the Fano of lines on these cubic fourfolds are actually isomorphic to Hilbert schemes and not only birational, however it is at the same time weaker in the sense that we can obtain density also by fixing a degree $2e_0$ for the $K3$ without letting it vary.} 
\end{rmk}

 \begin{rmk}
 \em{By Proposition \ref{prop:hass}, for  all integers $d \not= 2\frac{m^2+m+1}{
a^2}$ the euclidian density given by Theorem \ref{thm:hass++} is the best result one can obtain.}
 \end{rmk}






\end{document}